\documentclass[twoside]{amsart}
\usepackage{amsfonts, amsthm, amssymb}
\usepackage{enumerate}
\usepackage{amsmath}

\theoremstyle{plain}
\newtheorem*{thmA}{Theorem~A}

\newtheorem*{acknowledgements}{Acknowledgements}

\newtheorem*{MT1}{Theorem~1}
\newtheorem*{MT2}{Theorem~2}
\newtheorem*{MT3}{Theorem~3}

\newtheorem*{proposiotion A}{Proposition~A}
\newtheorem*{lemma 6.1-(1)}{Lemma~6.1*}

\newtheorem{theorem}{Theorem}[section]

\newtheorem{lemma}[theorem]{Lemma}
\newtheorem{proposition}[theorem]{Proposition}
\newtheorem{remark}[theorem]{Remark}

\def \N{\nabla}

\def \x{\xi}

\def \x{\xi}

\theoremstyle{plain}

\begin{document}

\title[real hypersurface of type~$\rm B$]{Commuting Jacobi operators on Real hypersurfaces of Type~$\rm B$ in the complex quadric}

\vspace{0.2in}
\author[H. Lee and Y.J. Suh]{Hyunjin Lee and Young Jin Suh}

\address{\newline
Hyunjin Lee
\newline The Research Institute of Real and Complex Manifolds (RIRCM),
\newline Kyungpook National University,
\newline Daegu 41566, REPUBLIC OF KOREA}
\email{lhjibis@hanmail.net}

\address{\newline
Young Jin Suh
\newline Department of Mathematics \& RIRCM,
\newline Kyungpook National University,
\newline Daegu 41566, REPUBLIC OF KOREA}
\email{yjsuh@knu.ac.kr}


\begin{abstract}
In this paper, first, we investigate the commuting property between the normal Jacobi operator~${\bar R}_N$ and the structure Jacobi operator~$R_{\xi}$ for Hopf real hypersurfaces in the complex quadric~$Q^m = SO_{m+2}/SO_mSO_2$, $m \geq 3$, which is defined by ${\bar R}_N R_{\xi} = R_{\xi}{\bar R}_N$. Moreover, a new characterization of Hopf real hypersurfaces with $\mathfrak A$-principal singular normal vector field in the complex quadric~$Q^{m}$ is obtained. By virtue of this result, we can give a remarkable classification of Hopf real hypersurfaces in the complex quadric~$Q^{m}$ with commuting Jacobi operators.
\end{abstract}

\maketitle

\section{Introduction}\label{section 1}
\setcounter{equation}{0}
\renewcommand{\theequation}{1.\arabic{equation}}
\vspace{0.13in}

In the class of Hermitian symmetric spaces of rank~2, usually we can give examples of Riemannian symmetric spaces $G_{2}(\mathbb C^{m+2})= SU_{m+2}/S(U_2U_m)$ and $G_{2}^{*}(\mathbb C^{m+2})=SU_{2,m}/S(U_2U_m)$, which are said to be complex two-plane Grassmannians and complex hyperbolic two-plane Grassmannians, respectively (see \cite{BLS2013}, \cite{K2}, \cite{SH}, \cite{S} and \cite{S5}). These are viewed as Hermitian symmetric spaces and quaternionic K\"{a}hler symmetric spaces equipped with the K\"{a}hler structure $J$ and the quaternionic K\"{a}hler structure ${\mathfrak J}$. There are exactly two types of singular tangent vectors $X$ of complex $2$-plane Grassmannians $G_{2}(\mathbb C^{m+2})$ and complex hyperbolic $2$-plane Grassmannians
$G_{2}^{*}(\mathbb C^{m+2})$ which are characterized by the geometric properties $JX \in {\mathfrak J}X$ and $JX \perp {\mathfrak J}X$ respectively.

\vskip 6pt

As another kind of Hermitian symmetric space with rank $2$ of compact type different from the above ones, we can give the example of complex quadric $Q^m = SO_{m+2}/SO_mSO_2$, which is a complex hypersurface in complex projective space ${\mathbb C}P^{m+1}$ (see \cite{R}, \cite{R1}, \cite{R2},  \cite{BS} and \cite{S1}). The complex quadric also can be regarded as a kind of real Grassmann manifold of compact type with rank~2 (see \cite{He} and \cite{KO}). Accordingly, the complex quadric admits both a complex conjugation structure~$A$ and a K\"ahler structure~$J$, which anti-commutes with each other, that is, $AJ=-JA$. Then for $m \geq 3$ the triple $(Q^m,J,g)$ is a Hermitian symmetric space of compact type with rank~2 and its maximal sectional curvature is equal to~$4$ (see \cite{K} and \cite{R}).

\vskip 6pt

In addition to the complex structure~$J$ there is another distinguished geometric structure on~${Q}^m$, namely a parallel rank two vector bundle~${\mathfrak A}$ which contains an $S^1$-bundle of real structures, that is, complex conjugations~$A$ on the tangent spaces of $Q^m$. The set is denoted by ${\mathfrak A}_{[z]}=\{A_{{\lambda}\bar z} \, \vert \, {\lambda} \in S^1 \subset \mathbb C \}$, $[z] \in Q^m$, and it is the set of all complex conjugations defined on $Q^m$. Then ${\mathfrak A}_{[z]}$ becomes a parallel rank $2$-subbundle of $\mathrm{End}(TQ^m)$. This geometric structure determines a maximal ${\mathfrak A}$-invariant subbundle ${\mathcal Q}$ of the tangent bundle $TM$ of a real hypersurface $M$ in~${Q}^m$.  Here the notion of parallel vector bundle ${\mathfrak A}$ means that $({\bar\nabla}_XA)Y=q(X)JAY$ for any vector fields $X$ and $Y$ on $Q^m$, where  $\bar\nabla$ and $q$ denote a connection and a certain $1$-form defined on $T_{[z]}Q^m$, $[z] \in Q^m$, respectively (see \cite{BS}).

\vskip 6pt

Recall that a nonzero tangent vector $W \in T_{[z]}Q^m$ is called {\it singular} if it is tangent to more than one maximal flat in $Q^m$. There are two types of singular tangent vectors for the complex quadric $Q^m$:
\begin{enumerate}[\rm (a)]
\item {If there exists a conjugation $A \in {\mathfrak A}$ such that $W \in V(A):= \{W\,|\, AW=W \}$, then $W$ is singular. Such a singular tangent vector is called {\it ${\mathfrak A}$-principal}.}
\item {If there exist a conjugation $A \in {\mathfrak A}$ and orthonormal vectors $X,Y \in V(A)$ such that $W/||W|| = (X+JY)/\sqrt{2}$, then $W$ is singular. Such a singular tangent vector is called {\it ${\mathfrak A}$-isotropic}.}
\end{enumerate}

\vskip 3pt

On the other hand, a typical characterization for real hypersurfaces with the $\mathfrak A$-principal normal vector field in $Q^{m}$ was introduced in \cite{BS2015} as follows.
\begin{thmA}\label{Theorem A}
Let $M$ be a connected orientable real hypersurface with constant mean curvature in the complex quadric~$Q^m$, $m\geq 3$. Then $M$ is a contact hypersurface if and only if $M$ is congruent to an open part of the around the $m$-dimensional sphere~$S^{m}$ which is embedded in $Q^{m}$ as a real form of $Q^{m}$.
\end{thmA}
\noindent Actually, we say that $M$ is a {\it contact} hypersurface of a Kaehler manifold if there exists an everywhere nonzero smooth function $\rho$ such that $d \eta (X,Y) = 2 \rho g(\phi X, Y)$ holds on $M$. Here $(\phi, \xi, \eta, g)$ is an almost contact metric structure of $M$. It can be easily verified that a real hypersurface $M$ is contact if and only if there exists an everywhere nonzero constant function $\rho$ on $M$ such that $S \phi + \phi S = 2 \rho \phi$. In particular, this concept of contact real hypersurfaces can be regarded as a typical characterization of model spaces of type~$B$ in complex projective space and complex hyperbolic space, respectively (see \cite{Kon} and \cite{Ver}). To our knowledge this is the only characterization of {\it the model space of type~$B$} in $Q^{m}$, which is {\it the tube around $m$-dimensional sphere~$S^{m}$ in $Q^{m}$} (Hereafter we denote this model space $(\mathcal T_{B})$).

\vskip 6pt

In this paper, we investigate some characterization problem for Hopf real hypersurfaces in $Q^{m}$. The notion of {\it Hopf} means that the Reeb vector field~$\xi$ of~$M$ is principal by the shape operator~$S$ of $M$, that is, $S\xi=g(S\xi, \xi) \xi = \alpha \xi$. When the Reeb curvature function~$\alpha =g(S\xi, \xi)$ is identically vanishing, we say that $M$ has a {\it vanishing geodesic Reeb flow}. Otherwise, $M$ has a {\it non-vanishing geodesic Reeb flow}. Recently, many characterizations of Hopf real hypersurfaces in the complex quadric~$Q^{m}$ have been given by some differential geometers from various geometric view points (see \cite{BeSuh}, \cite{BS2015}, \cite{JS2019}, \cite{KS2019}, \cite{K}, \cite{LS2018}, \cite{P2018}, \cite{PJKS2018}, \cite{PS2018}  etc).

\vskip 6pt

On the other hand, Jacobi fields along geodesics of a given Riemannian manifold~$(\widetilde M, \widetilde g)$ satisfy a well known differential equation (see \cite{Carmo}). This equation naturally inspires the so-called Jacobi operator. That is, if $\widetilde R$ denotes the curvature operator of $\widetilde M$, and $Z$ is tangent vector field to $\widetilde M$, then the Jacobi operator $\widetilde{R}_Z \in \mathrm{End}(T_p \widetilde{M})$ with respect to $Z$ at $p \in \widetilde M$, defined by $(\widetilde{R}_Z Y)(p)=(\widetilde{R}(Y,Z)Z)(p)$ for any $Z \in T_p \widetilde{M}$, becomes a self-adjoint endomorphism of the tangent bundle~$T \widetilde{M}$ of $\widetilde M$. Thus, the normal vector field~$N$ of a real hypersurface~$M$ in $Q^{m}$ provides the Jacobi operator~${\bar R}_N \in \mathrm{End}(TM)$ called by {\it normal Jacobi operator}. Moreover for the Reeb vector field $\xi:=-JN \in TM$ the Jacobi operator $R_{\xi} \in \mathrm{End}(TM)$ is said to be a {\it structure Jacobi operator}. Here $\bar R$ and $R$ are the Riemannian curvature tensors for $Q^{m}$ and its real hypersurface~$M$, respectively.

\vskip 6pt

By the Kaehler structure~$J$ of the complex quadric~$Q^{m}$, we can decompose its action on any tangent vector field~$X$ on $M$ in $Q^{m}$ as follows:
\begin{equation*}
JX = \phi X + \eta(X) N,
\end{equation*}
where $\phi X$ denotes the tangential component of $JX$ and $\eta$ denotes a 1-form defined by $\eta(X)=g(JX, N)=g(X,\xi)$ for the Reeb vector field $\xi=-JN$ and $N$ a unit normal vector field on $M$ in $Q^{m}$. When the Ricci tensor $\mathfrak{Ric}$ of $M$ in $Q^{m}$ commutes with the structure tensor $\phi$, that is, $\mathfrak{Ric}\, \phi = \phi \, \mathfrak {Ric}$, we say that $M$ has {\it Ricci commuting} or {\it commuting Ricci tensor}. P\'{e}rez and Suh~\cite{PS} proved a non-existence property for Hopf real hypersurfaces in $G_{2}({\mathbb C}^{m+2})$ with parallel and commuting Ricci tensor. In \cite{SHw} Suh and Hwnag gave another classification for real hypersurfaces in $Q^{m}$ with commuting Ricci tensor. Recently, in~\cite{SLW} the present authors and Woo studied the commuting normal Jacobi operator (resp. the structure Jacobi operator) defined by ${\bar R}_{N} \phi = \phi {\bar R}_{N}$ (resp. $R_{\xi} \phi = \phi R_{\xi}$).

\vskip 6pt

{\it Motivated by these studies, in this paper, we consider the commuting property between the normal Jacobi operator~${\bar R}_{N}$ and structure Jacobi operator~$R_{\xi}$ given by}
\begin{equation}\label{C1}
{\bar R}_{N} R_{\xi} = R_{\xi} {\bar R}_{N}.
\tag{*}
\end{equation}
Actually, the study for the commuting property with Jacobi operators was first initiated by Brozos-V\'{a}zquex and Gilkey~\cite{BVG}. They gave two results for a Riemannian manifold $({\widetilde M}^{m}, \widetilde g)$, $m \geq 3$, as follows:  One is: if ${\widetilde R}_{U} {\widetilde R}_{V} = {\widetilde R}_{V} {\widetilde R}_{U}$ for all tangent vector fields~$U, V$ on ${\widetilde M}$, then $\widetilde M$ is flat. The other is : if the same occurs for any $U \bot V$, then $\widetilde M$ has constant sectional curvature. In addition, in ~\cite{MPS} the authors classified real hypersurfaces in $G_{2}(\mathbb C^{m+2})$ whose structure Jacobi operator commutes either with the normal Jacobi operator. Now in this paper, first, we prove that our commuting property~\eqref{C1} is equivalent to the singularity of normal vector field for a Hopf real hypersurface in~$Q^{m}$ as follows:
\begin{MT1}
Let $M$ be a real hypersurface with non-vanishing geodesic Reeb flow in $Q^{m}$, $m \geq 3$. Then $M$ has the $\mathfrak A$-principal normal vector field if and only if the normal Jacobi operator ${\bar R}_{N}$ commutes with the structure Jacobi operator~$R_{\xi}$.
\end{MT1}

\noindent Related to Theorem~1, naturally, some characterizations of Hopf hypersurfaces in terms of singularity of the normal vector field are being investigated. Among them, as a new characterization of $\mathfrak A$-principal singular normal, we can give one of remarkable results as follows:
\begin{MT2}
Let $M$ be a Hopf real hypersurface in $Q^{m}$, $m \geq 3$. Then $M$ has the $\mathfrak A$-principal normal vector field if and only if $M$ is locally congruent to the model space of type $(\mathcal T_{B})$, that is, a tube over $m$-dimensional sphere~$S^{m}$ in $Q^{m}$.
\end{MT2}

\noindent By virtue of Theorems~1 and 2, we also assert the following: {\it Let $M$ be a real hypersurface with non-vanishing geodesic Reeb flow in~$Q^{m}$, $m \geq 3$. Then $M$ has the commuting normal Jacobi operator, ${\bar R}_{N} R_{\xi} = R_{\xi} {\bar R}_{N}$, if and only if $M$ is locally congruent to the model space of type $(\mathcal T_{B})$}. Motivated by this result, we can give another remarkable result related to commuting Jacobi operators as follows:
\begin{MT3}
Let $M$ be a Hopf real hypersurface in~$Q^{m}$, $m \geq 3$. Then $M$ has the commuting normal Jacobi operator, ${\bar R}_{N} R_{X} = R_{X} {\bar R}_{N}$ for all tangent vector fields~$X \in \mathcal C = \{X \in TM\, |\, X \bot \xi\}$ if and only if $M$ is locally congruent to the model space of type $(\mathcal T_{B})$.
\end{MT3}

\vskip 17pt

\section{The complex quadric}\label{section 2}
 \setcounter{equation}{0}
\renewcommand{\theequation}{2.\arabic{equation}}
\vspace{0.13in}

For more background to this section we refer to \cite{K}, \cite{K2}, \cite{KO} and \cite{R}. The complex quadric $Q^m$ is the complex hypersurface in ${\mathbb C}P^{m+1}$ which is defined by the equation $z_1^2 + \cdots + z_{m+2}^2 = 0$, where $z_1,\cdots,z_{m+2}$ are homogeneous coordinates on ${\mathbb C}P^{m+1}$. We equip $Q^m$ with the Riemannian metric which is induced from the Fubini Study metric on ${\mathbb C}P^{m+1}$ with constant holomorphic sectional curvature~$4$. The K\"{a}hler structure on ${\mathbb C}P^{m+1}$ induces canonically a K\"{a}hler structure $(J,g)$ on the complex quadric. For a nonzero vector $z \in \mathbb C^{m+1}$ we denote by $[z]$ the complex span of $z$, that is, $[z]=\mathbb C z = \{\lambda z\,|\, \lambda \in S^{1} \subset \mathbb C \}$. Note that by definition~$[z]$ is a point in $\mathbb C P^{m+1}$. For each $[z] \in Q^m \subset \mathbb C P^{m+1}$ we identify $T_{[z]}{\mathbb C}P^{m+1}$ with the orthogonal complement ${\mathbb C}^{m+2} \ominus {\mathbb C}z$ of ${\mathbb C}z$ in ${\mathbb C}^{m+2}$ (see Kobayashi and Nomizu~\cite{KO}). The tangent space $T_{[z]}Q^m$ can then be identified canonically with the orthogonal complement ${\mathbb C}^{m+2} \ominus ({\mathbb C}z \oplus {\mathbb C}\rho)$ of ${\mathbb C}z \oplus {\mathbb C}\rho$ in ${\mathbb C}^{m+2}$, where $\rho \in \nu_{[z]}Q^m$ is a normal vector of $Q^m$ in ${\mathbb C}P^{m+1}$ at the point $[z]$.

\vskip 6pt

The complex projective space ${\mathbb C}P^{m+1}$ is a Hermitian symmetric space of the special unitary group $SU_{m+2}$, namely ${\mathbb C}P^{m+1} = SU_{m+2}/S(U_{m+1}U_1)$. We denote by $o = [0,\ldots,0,1] \in {\mathbb C}P^{m+1}$ the fixed point of the action of the stabilizer $S(U_{m+1}U_1)$. The special orthogonal group $SO_{m+2} \subset SU_{m+2}$ acts on ${\mathbb C}P^{m+1}$ with cohomogeneity one. The orbit containing $o$ is a totally geodesic real projective space ${\mathbb R}P^{m+1} \subset {\mathbb C}P^{m+1}$. The second singular orbit of this action is the complex quadric $Q^m = SO_{m+2}/SO_mSO_2$. This homogeneous space model leads to the geometric interpretation of the complex quadric $Q^m$ as the Grassmann manifold $G_2^+({\mathbb R}^{m+2})$ of oriented $2$-planes in ${\mathbb R}^{m+2}$. It also gives a model of $Q^m$ as a Hermitian symmetric space of rank $2$. The complex quadric $Q^1$ is isometric to a sphere $S^2$ with constant curvature, and $Q^2$ is isometric to the Riemannian product of two $2$-spheres with constant curvature. For this reason we will assume $m \geq 3$ from now on.

\vskip 6pt

For a unit normal vector $\rho$ of $Q^m$ at a point $[z] \in Q^m$ we denote by $A = A_\rho$ the shape operator of $Q^m$ in ${\mathbb C}P^{m+1}$ with respect to $\rho$. The shape operator is an involution on the tangent space $T_{[z]}Q^m$ and
$$
T_{[z]}Q^m = V(A_\rho) \oplus JV(A_\rho),
$$
where $V(A_\rho)$ is the $(+1)$-eigenspace and $JV(A_\rho)$ is the $(-1)$-eigenspace of $A_\rho$.  Geometrically this means that the shape operator $A_\rho$ defines a real structure on the complex vector space $T_{[z]}Q^m$, or equivalently, is a complex conjugation on $T_{[z]}Q^m$. Since the real codimension of $Q^m$ in ${\mathbb C}P^{m+1}$ is $2$, this induces an $S^1$-subbundle ${\mathfrak A}$ of the endomorphism bundle ${\rm End}(TQ^m)$ consisting of complex conjugations. There is a geometric interpretation of these conjugations. The complex quadric~$Q^m$ can be viewed as the complexification of the $m$-dimensional sphere~$S^m$. Through each point $[z] \in Q^m$ there exists a one-parameter family of Lagrangian submanifolds in $Q^m$ which are isometric to the sphere $S^m$. These real forms are congruent to each other under action of the center $SO_2$ of the isotropy subgroup of $SO_{m+2}$ at $[z]$. The isometric reflection of $Q^m$ in such a real form $S^m$ is an isometry, and the differential at $[z]$ of such a reflection is a conjugation on $T_{[z]}Q^m$. In this way the family ${\mathfrak A}$ of conjugations on $T_{[z]}Q^m$ corresponds to the family of real forms $S^m$ of $Q^m$ containing $[z]$, and the subspaces $V(A) \subset T_{[z]}Q^m$ correspond to the tangent spaces $T_{[z]}S^m$ of the real forms~$S^m$ of $Q^m$.

\vskip 6pt

The Gauss equation for $Q^m \subset {\mathbb C}P^{m+1}$ implies that the Riemannian curvature tensor $\bar R$ of $Q^m$ can be described in terms of the complex structure $J$ and the complex conjugation $A \in {\mathfrak A}$:
\begin{equation}\label{Riemannian curvature tensor}
\begin{split}
{\bar R}(X,Y)Z & =  g(Y,Z)X - g(X,Z)Y + g(JY,Z)JX - g(JX,Z)JY \\
& \quad  - 2g(JX,Y)JZ  + g(AY,Z)AX \\
& \quad - g(AX,Z)AY + g(JAY,Z)JAX - g(JAX,Z)JAY.
\end{split}
\end{equation}
By using the Gauss and Wingarten formulas the left-hand side of \eqref{Riemannian curvature tensor} becomes
\begin{equation*}
\begin{split}
{\bar R}(X,Y)Z  & = R(X,Y)Z -g(SY, Z)SX + g(SX, Z)SY \\
& \quad + \big\{g((\nabla_{X}S)Y, Z)- g((\nabla_{Y}S)X, Z) \big \} N,
\end{split}
\end{equation*}
where $R$ and $S$ denote the Riemannian curvature tensor and the shape operator of a real hypersurface $M$ in $Q^{m}$, respectively.

\vskip 6pt

\noindent From this, taking tangent and normal components respectively, we have
\begin{equation}\label{eq: 2.1}
\begin{split}
& g(R(X,Y)Z, W) - g(SY,Z)g(SX,W) + g(SX,Z)g(SY,W) \\
& =  g(Y,Z)g(X,W) - g(X,Z)g(Y,W) + g(JY,Z)g(JX,W) \\
& \ \ - g(JX,Z)g(JY,W) - 2g(JX,Y)g(JZ, W) + g(AY,Z)g(AX, W) \\
& \ \  - g(AX,Z)g(AY, W)+ g(JAY,Z)g(JAX,W)- g(JAX,Z)g(JAY,W), \\
\end{split}
\end{equation}
and
\begin{equation}\label{codazzi equation}
\begin{split}
& g((\nabla_{X}S)Y, Z) - g((\nabla_{Y}S)X, Z)\\
& =  \eta(X) g(JY,Z) - \eta(Y) g(JX,Z)  - 2 \eta(Z) g(JX,Y)  \\
& \quad + g(AY,Z)g(AX, N) - g(AX,Z)g(AY, N)  \\
& \quad + g(AX, \xi) g(JAY,Z)-  g(AY, \xi) g(JAX,Z).
\end{split}
\end{equation}

\vskip 6pt

It is well known that for every unit tangent vector $W \in T_{[z]}Q^m$ there exist a conjugation $A \in {\mathfrak A}$ and orthonormal vectors $Z_{1}$, $Z_{2} \in V(A)$ such that
\begin{equation*}
W = \cos (t) Z_{1} + \sin (t) JZ_{2}
\end{equation*}
for some $t \in [0,\pi/4]$  (see \cite{R}). The singular tangent vectors correspond to the values $t = 0$ and $t = \pi/4$. If $0 < t < \pi/4$ then the unique maximal flat containing~$W$ is ${\mathbb R}Z_{1} \oplus {\mathbb R}JZ_{2}$.

\vskip 17pt

\section{Some general equations}\label{section 3}
\setcounter{equation}{0}
\renewcommand{\theequation}{3.\arabic{equation}}
\vspace{0.13in}

Let $M$ be a  real hypersurface in $Q^m$ and denote by $(\phi,\xi,\eta,g)$ the induced almost contact metric structure. Note that $JX=\phi X + \eta(X)N$ and $JN=-\xi$, where $\phi X$ is the tangential component of $JX$ and $N$ is a (local) unit normal vector field of $M$. The tangent bundle $TM$ of $M$ splits orthogonally into  $TM = {\mathcal C} \oplus {\mathbb R}\xi$, where ${\mathcal C} = \mathrm{ker}\,\eta$ is the maximal complex subbundle of $TM$. The structure tensor field $\phi$ restricted to ${\mathcal C}$ coincides with the complex structure $J$ restricted to~${\mathcal C}$, and $\phi \xi = 0$. Moreover, since $Q^{m}$ has also a real structure~$A$, we decompose $AX$ into its tangential and normal components for a fixed $A \in \mathfrak A_{[z]}$ and $X \in T_{[z]}M$:
\begin{equation}\label{AX}
AX=BX + \rho(X)N
\end{equation}
where $BX$ is the tangential component of $AX$ and
\begin{equation*}
\rho(X)=g(AX, N)=g(X, AN)=g(X, AJ\xi) = g(JX, A \xi).
\end{equation*}

\vskip 6pt

At each point $[z] \in M$ we can choose $A \in {\mathfrak A}_{[z]}$ such that
\begin{equation*}
N = \cos (t) Z_1 + \sin (t) JZ_2
\end{equation*}
for some orthonormal vectors $Z_1$, $Z_2 \in V(A)$ and $0 \leq t \leq \frac{\pi}{4}$ (see Proposition~3 in~\cite{R}). Note that $t$ is a function on $M$. From this and $\xi = -JN$, we have
\begin{equation}\label{AN, Axi}
\begin{cases}
\xi  =  \sin (t) Z_2 - \cos (t) JZ_1, \\
AN  =  \cos (t) Z_1 - \sin (t) JZ_2,  \\
A\xi  =  \sin (t) Z_2 + \cos (t) JZ_1.
\end{cases}
\end{equation}
These formulas leads to $g(\xi,AN) = 0$ and $g(A\xi, \xi) = -g(AN, N)=-\cos ( 2t)$ on~$M$.

\vskip 6pt

We now assume that $M$ is a Hopf real hypersurface in $Q^{m}$. Then the shape operator~$S$ of~$M$ satisfies
$S\xi = \alpha \xi$ with the Reeb curvature function $\alpha = g(S\xi,\xi)$ on $M$. By virtue of the Codazzi equation, we obtain the following lemma.
\begin{lemma}[\cite{BeSuh}, \cite{SDGA}] \label{lemma Hopf}
Let $M$ be a Hopf real hypersurface in $Q^m$, $m \geq 3$. Then we obtain
\begin{equation}\label{eq: 3.2}
\begin{split}
X \alpha  = (\xi \alpha) \eta(X)  + 2g(A\xi,\xi)g(X,AN)\\
(\textrm{i.e.} \ \ \mathrm{grad}\,\alpha = (\xi \alpha) \xi - 2g(A\xi, \xi) \phi A \xi)
\end{split}
\end{equation}
and
\begin{equation}\label{eq: 3.1}
\begin{split}
& 2g(S \phi SX,Y) - \alpha g((\phi S + S\phi)X,Y) - 2 g(\phi X,Y) \\
& \quad  + 2g(X,AN)g(Y,A\xi) - 2g(Y,AN)g(X,A\xi)\\
& \quad  - 2g(X,AN)g(\xi,A\xi)\eta(Y) + 2g(Y,AN)g(\xi,A\xi)\eta(X)=0
\end{split}
\end{equation}
for any tangent vector fields $X$ and $Y$ on $M$.
\end{lemma}

\noindent In addition, if $M$ has a singular normal vector field~$N$, then the gradient of $\alpha$ should be $\mathrm{grad}\, \alpha = (\xi \alpha) \xi$. From the property of $g(\nabla_{X} \mathrm{grad}\, \alpha, Y) = g(\nabla_{Y} \mathrm{grad}\, \alpha, X)$ we obtain
\begin{equation}\label{e 3.5}
\big(X(\xi \alpha)\big)\eta(Y) + (\xi \alpha) g(\phi SX, Y) = \big(Y(\xi \alpha) \big)\eta(X) + (\xi \alpha) g(\phi SY, X)
\end{equation}
for all $X$, $Y \in TM$. Putting $Y=\xi$ in \eqref{e 3.5} it follows $\big(X(\xi \alpha)\big) = \big(\xi(\xi \alpha)\big)\eta(X)$. From this, the equation~\eqref{e 3.5} becomes
\begin{equation*}
(\xi \alpha) g\big((\phi S + S \phi )X, Y \big) = 0.
\end{equation*}
On the other hand, in \cite{LS2018} the authors gave that there does not any real hypersurface with the anti-commuting property, $S \phi + \phi S =0$, in $Q^{m}$, $m \geq 3$. By virtue of this result, we get $(\xi \alpha) = 0$. Then from this and \eqref{eq: 3.2}, we assert:
\begin{lemma}\label{lemma constant}
Let $M$ be a Hopf real hypersurface in $Q^{m}$, $m \geq 3$. If $M$ has a singular normal vector field, then the Reeb curvature function~$\alpha$ should be constant.
\end{lemma}

Specially, it has been known for a Hopf real hypersurface with $\mathfrak A$-principal normal vector field as follows:
\begin{lemma}[\rm \cite{S1}]\label{lemma 3.2}
Let $M$ be a Hopf real hypersurface in $Q^m$ such that the normal vector
field $N$ is ${\mathfrak A}$-principal everywhere. Then the Reeb curvature function~$\alpha$ is constant. Moreover, if $X \in {\mathcal C}$ is a principal curvature vector of $M$ with principal curvature $\lambda$, then $2\lambda \neq
\alpha$ and its corresponding vector~$\phi X$ is a principal curvature vector of $M$ with
principal curvature $\frac{\alpha\lambda + 2}{2{\lambda}-{\alpha}}$.
\end{lemma}

When the normal vector $N$ is $\mathfrak A$-isotropic, the tangent vector space~$T_{[z]}M$ at $[z] \in M$ is decomposed by
\begin{equation*}
T_{[z]}M=[{\xi}] \oplus [A{\xi},AN] \oplus {\mathcal Q}_{[z]},
\end{equation*}
where ${\mathcal C}_{[z]} \ominus {\mathcal Q}_{[z]}={\mathcal Q}_{[z]}^{\bot}=\text{Span}[A{\xi},AN]$. For this decomposition we obtain:
\begin{lemma}[\cite{LS2018}]\label{lemma 3.5}
Let $M$ be a Hopf hypersurface in $Q^m$ such that the normal vector
field $N$ is ${\mathfrak A}$-isotropic. Then $S A \xi = 0$ and $SAN =0$. Moreover, if $X \in {\mathcal Q}$ is a principal curvature
vector of $M$ with principal curvature $\lambda$, then $2\lambda \neq
\alpha$ and its corresponding vector~$\phi X$ is a principal curvature vector of $M$ with
principal curvature $\frac{\alpha\lambda + 2}{2{\lambda}-{\alpha}}$.
\end{lemma}

On the other hand, from the property of $g(A\xi, N)=0$ on a real hypersurface~$M$ in $Q^{m}$ we see that the non-zero vector field~$A\xi$ is tangent to $M$. Hence by Gauss formula, ${\bar \nabla}_{U}V = \nabla_{U}V + \sigma(U,V)$ for $U$, $V \in TM$, it induces
\begin{equation}\label{covariant derivative of Axi}
\begin{split}
\N_{X}(A\x)& = {\bar \N}_{X}(A\x) - \sigma(X, A\x)  \\
& =q(X) JA\x + A(\N_{X}\x) + g(SX, \xi) AN - g(SX, A \x)N
\end{split}
\end{equation}
for any $X \in TM$. Taking the inner product with $N$, we obtain
\begin{equation}\label{e 3.6}
q(X) g(A\xi, \xi) = - g(AN, \N_{X}\x) + g(SX, \xi) g(A\xi, \xi) + g (SX, A\xi)
\end{equation}
by using $g(AN, N)=-g(A\xi, \xi)$. In particular, if $M$ is Hopf, then this equation becomes
\begin{equation}\label{eq: 3.5}
q(\x) g(A\xi, \xi) = 2\alpha g(A\xi, \xi).
\end{equation}

\vskip 17pt

\section{Commuting Jacobi operator}\label{section 4}
\setcounter{equation}{0}
\renewcommand{\theequation}{4.\arabic{equation}}
\vspace{0.13in}

Now, we consider the commuting condition with respect to the normal Jacobi operator~${\bar R}_{N}$ and the structure Jacobi operator~$R_{\xi}$ on a Hopf real hypersurface~$M$ in complex quadrics~$Q^{m}$, $m \geq 3$. The Jacobi operator~${\bar R}_{N}\in \mathrm{End}(TQ^{m})$ with respect to the unit tangent vector~$N \in T_{[z]}Q^{m}$, $[z] \in Q^{m}$, is induced from the curvature tensor~${\bar R}$ of~$Q^{m}$ given in section~\ref{section 2} as follows:
\begin{equation*}\label{normal Jacobi operator}
\begin{split}
{\bar R}_{N} U &={\bar R}(U, N)N \\
               &= U -g(U, N)N + 3g(U, \xi)\xi + g(AN, N) AU  \\
               & \quad \  - g(AN, U) AN - g(A\xi, U) A \xi
\end{split}
\end{equation*}
for all vector field $U \in T Q^{m}$. Since $TQ^{m}= TM \oplus \mathrm{span}\{N\}$, we obtain
$$
{\bar R}_{N}Y = ({\bar R}_{N} Y)^{\top} + ({\bar R}_{N}Y)^{\bot}
$$
and
$$
({\bar R}_{N}Y)^{\bot} = g({\bar R}_{N}Y, N)N = g(\bar R(Y, N)N, N)=0
$$
for any vector field~$Y \in TM \subset TQ^{m}$. Hence ${\bar R}_{N} \in \mathrm{End}(TM)$ is defined by
\begin{equation}\label{normal Jacobi operator of M}
\begin{split}
{\bar R}_{N} Y &={\bar R}(Y, N)N \\
               &= Y + 3\eta(Y)\xi + g(AN, N) BY  + g(AN, Y) \phi A \xi - g(A\xi, Y) A \xi
\end{split}
\end{equation}
for all vector field $Y \in T M$. Here we have used \eqref{eq: 3.1} and $AN=AJ\xi =-JA\xi = -\phi A \xi - g(A\xi, \xi)N$.

\vskip 6pt

On the other hand, the structure Jacobi operator $R_{\xi}$ from \eqref{codazzi equation} can be rewritten as follows:
\begin{equation*}
\begin{split}
g(R_{\xi}Y, W) & =g(R(Y,{\xi}){\xi}, W)\\
&=g(Y, W)-{\eta}(Y)\eta(W)+{\beta} g(AY, W)-g(AY,{\xi})g(A{\xi}, W) \\
&\quad \  - g(AY,N)g(AN, W) +{\alpha}g(SY, W) - \alpha^{2}\eta(Y)\eta(W),
\end{split}
\end{equation*}
where we have put ${\alpha}=g(S{\xi},{\xi})$ and ${\beta}=g(A{\xi},{\xi})$, because we assume that $M$ is Hopf.  The Reeb vector field ${\xi}=-JN$ and the anti-commuting property $AJ=-JA$ gives ${\beta}=-g(AN,N)$. When this function ${\beta}=g(A{\xi},{\xi})$ identically vanishes, we say that a real hypersurface $M$ in $Q^m$ is $\mathfrak A$-isotropic as in section~\ref{section 1}. From this equation, we get the structure operator~$R_{\xi} \in \mathrm{End}(TM)$ as follows:
\begin{equation}\label{structure Jacobi op}
\begin{split}
R_{\xi}Y & = Y -{\eta}(Y)\xi + {\beta} BY -g(A\xi, Y)A{\xi} \\
&\quad \  - g(\phi A \xi, Y) \phi A \xi + {\alpha} SY - \alpha^{2}\eta(Y)\xi.
\end{split}
\end{equation}

\vskip 6pt

\def\RN{{\bar R}_{N}}
\def\Rx{R_{\xi}}

\noindent By the linearity of $\RN$ and $\Rx$, the commuting condition~\eqref{C1}, that is, $({\bar R}_{N}R_{\xi}) Y  = (R_{\xi}{\bar R}_{N}) Y$ for any $Y \in TM$, becomes
\begin{equation}\label{eq 4.3}
\begin{split}
& \Rx Y - \beta B(\Rx Y)  - g(\phi A \xi, \Rx Y) \phi A \xi  - (A \xi, \Rx Y) A\xi \\
& \ \      = \RN Y - \eta(\RN Y) \xi + \beta B(\RN Y) - g(A \xi, \RN Y) A\xi \\
&  \quad \ - g(\phi A \xi, \RN Y) \phi A \xi + \alpha S (\RN Y) - \alpha^{2} \eta(\RN Y) \xi
\end{split}
\end{equation}
together with \eqref{normal Jacobi operator of M} and \eqref{structure Jacobi op}.

\vskip 6pt

Using this condition, now  let us prove that the unit normal vector field~$N$ of $M$ is singular: $N$ is $\mathfrak A$-principal or $\mathfrak A$-isotropic.

\vskip 3pt

Substituting $Y =\xi$ in \eqref{eq 4.3}, then it yields
\begin{equation}\label{eq 4.4}
\begin{split}
0 & = \RN \xi - \eta(\RN \xi) \xi + \beta B(\RN \xi) - g(A \xi, \RN \xi) A\xi \\
&  \quad \ \ - g(\phi A \xi, \RN \xi) \phi A \xi + \alpha S (\RN \xi) - \alpha^{2} \eta(\RN \xi) \xi \\
& = 2 \alpha \beta (\alpha \beta \xi - SA\xi).
\end{split}
\end{equation}
where we have used $\Rx \xi =0$, $\RN \xi = 4 \xi - 2 \beta A\xi$ and $A^{2}=I$.

\vskip 3pt

By virtue of Remark~3.3 in \cite{LS2018} we see that if the Reeb curvature function~$\alpha$ is vanishing, then the normal vector field~$N$ of $M$ is singular. Hence, from now on, we only investigate the case of $\alpha \neq 0$. Then \eqref{eq 4.4} gives us the following two cases:

\vskip 3pt

{\bf Case I.} \ \  $\beta = g(A\xi, \xi) = 0$

\vskip 3pt

From the result of Reckziegel~\cite{R}, we obtain that $g(\xi, AN)=0$ and $\beta = g(A\xi, \xi) = -g(AN, N)=-\cos ( 2t)$, $t\in[0, \pi/4]$, on $M$ (see \eqref{AN, Axi} in section~\ref{section 3}). It leads that $t= \frac{\pi}{4}$, which means that the normal vector field~$N$ is $\mathfrak A$-isotropic.

\vskip 6pt

{\bf Case II.} \ \  $\beta \neq g(A\xi, \xi) = 0$ (that is, $SA \xi = \alpha \beta \xi$,  $\alpha \beta \neq 0$)

\vskip 3pt

From \eqref{eq: 3.1}, the assumption~$SA \xi = \alpha \beta \xi$ leads to
\begin{equation}\label{eq 4.5}
\alpha S \phi A \xi = - 2 g^{2}(A\xi, \xi) \phi A \xi= - 2 \beta^{2} \phi A \xi.
\end{equation}

Taking the covariant derivative for our assumption $SA \xi = \alpha \beta \xi$ along any tangent vector~$X$ of $M$, we have
\begin{equation}\label{eq 4.6}
\begin{split}
&(\nabla_{X}S)A \xi + S (\nabla_{X}(A \xi)) \\
& \quad \ \ = (X \alpha) \beta \xi + \alpha g( \nabla_{X}(A\xi),  \xi)\xi + \alpha g(A \xi, \nabla_{X} \xi) \xi + \alpha \beta \nabla_{X} \xi.
\end{split}
\end{equation}
In addition, taking an inner product of \eqref{eq 4.6} with $Y \in T_{[z]}M$, $[z] \in M$, it yields
\begin{equation*}
\begin{split}
& g((\nabla_{Y}S)X, A\xi) - 2 \beta g(\phi X, Y) + q(X) g(\phi A \xi, SY) \\
& \quad \ \ + g(\phi SX, A SY) - \alpha \eta(X) g(\phi A \xi, SY) \\
& = (X \alpha) \beta \eta(Y)  + \alpha g(A \xi, \phi SX) \eta(Y),
\end{split}
\end{equation*}
together with \eqref{codazzi equation} and \eqref{covariant derivative of Axi}. Moreover, by \eqref{eq: 3.2} and \eqref{eq 4.5}, it follows
\begin{equation}\label{eq 4.7}
\alpha g((\nabla_{Y}S)X, A\xi) - 2 \alpha \beta g(\phi X, Y) + \alpha g(\phi SX, ASY) =0 \end{equation}
for any $X \in \mathcal C=\{X \in TM | X \bot \xi \}$ and $Y \in TM$.

Substituting $X=\phi A\xi \in \mathcal C$ and $Y=\xi$ and using \eqref{eq 4.5} this equation becomes
\begin{equation}\label{eq 4.8}
\alpha g((\nabla_{\xi}S)\phi A \xi, A\xi) + 2 \alpha \beta^{2} (1 - \beta^{2})=0,
\end{equation}
where we have used $g(\phi A \xi, \phi A \x)= 1- g^{2}(A\xi, \xi)= 1- \beta^{2}$. On the other hand, from our assumption $SA \xi = \alpha \beta \xi$ we obtain $\alpha g((\nabla_{\xi}S)\phi A \xi, A\xi) = \alpha g((\nabla_{\xi}S) A \xi, \phi A \xi) = -2 \beta^{2} (q(\xi) - \alpha) ( 1-\beta^{2})$. Hence \eqref{eq 4.8} yields
\begin{equation*}\label{eq 4.9}
0= -2\beta^{2} q(\xi) (1- \beta^{2}) = 4 \alpha \beta^{2} (1-\beta^{2}),
\end{equation*}
where we have used \eqref{eq: 3.5} in the second equality. Since $\alpha \beta \neq 0$, it yields $\beta=-\cos 2t = \pm 1$, $t \in [0, \frac{\pi}{4})$. So, we get $t =0$, which means that the normal vector field~$N$ of $M$ should be $\mathfrak A$-principal. Thus the proof of the following proposition is completed.
\begin{proposition}\label{proposition 4.1}
Let $M$ be a Hopf real hypersurface with commuting Jacobi operators such that ${\bar R}_{N} R_{\xi}=R_{\xi} {\bar R}_{N}$ in the complex quadric~$Q^{m}$, $m \geq 3$. Then the unit normal vector field~$N$ should be singular. It means that $N$ becomes either $\mathfrak A$-isotropic or $\mathfrak A$-principal.
\end{proposition}

\vskip 17pt

\section{Proof of Theorem~1}\label{section 5}
\setcounter{equation}{0}
\renewcommand{\theequation}{5.\arabic{equation}}
\vspace{0.13in}

In this section, we will give a proof of Theorem~$1$. Let $M$ be a real hypersurface with non-vanishing geodesic Reeb flow, $\alpha = g(S \xi, \xi) \neq 0$, in $Q^{m}$, $m \geq 3$. As mentioned in section~\ref{section 4}, if $M$ satisfies the commuting condition~\eqref{C1}, the normal vector field~$N$ of $M$ should be singular. Then from the definition of singular tangent vector field of $Q^{m}$, we can divided the following two cases.

\subsection{Commuting Jacobi operator with $\mathfrak A$-isotropic unit normal vector field}
Let us consider the case that the unit normal vector field~$N$ of $M$ is $\mathfrak A$-isotropic singular. It means that the normal vector field $N$ can be expressed by
$$N=\frac{1}{\sqrt 2}(Z_1+JZ_2)$$
for $Z_1$, $Z_2 \in V(A)$, where $V(A)$ denotes a $(+1)$-eigenspace of the complex conjugation~$A \in~{\mathfrak A}$. Then it follows that
$$AN=\frac{1}{\sqrt 2}(Z_1-JZ_2),\ AJN=-\frac{1}{\sqrt 2}(JZ_1+Z_2),\  \text{and}\  JN=\frac{1}{\sqrt 2}(JZ_1-Z_2).$$
Thus, we obtain that
$$g({\xi},A{\xi})=g(JN,AJN)=0,\  g({\xi},AN)=0\ \text{and}\ g(AN,N)=0,$$
which means that two vector fields $AN$ and $A\xi$ are tangent on $M$. By virtue of these formulas with respect to $\mathfrak A$-isotropic unit normal vector field and $g(JAY,{\xi})=-g(AY,J{\xi})=-g(AY,N)$, the normal Jacobi operator~${\bar R}_{N}$ and the structure Jacobi operator~$R_{\xi}$ can be rearranged as follows:
 \begin{equation*}
{\bar R}_{N}Y = Y + 3 \eta(Y) \xi - g(AN, Y)AN - g(A \xi, Y) A \xi
\end{equation*}
and
\begin{equation*}\label{eq 5.1}
R_{\xi}Y = Y- \eta(Y) \xi - g(A \xi, Y) A\xi - g(AN, Y) AN + \alpha SY - \alpha^{2} \eta(Y) \xi
\end{equation*}
respectively. So, the property for the commuting Jacobi operator $(R_{\xi} {\bar R}_{N}) = ({\bar R}_{N} R_{\xi})$ on $M$ is equivalent to
\begin{equation*}
\alpha SY = -6 \alpha^{2} \eta(Y) \xi,
\end{equation*}
where we have used
\begin{equation*}
\left\{
\begin{aligned}
& {\bar R}_{N}(A\xi)={\bar R}_{N}(AN)=0, \quad  {\bar R}_{N}\xi = 4 \xi, \\
& R_{\xi}\xi = R_{\xi}(AN)=R_{\xi}(A \xi) =0.
\end{aligned}
\right.
\end{equation*}
It gives us
\begin{equation}\label{eq 5.2}
SY = - 6 \alpha \eta(Y) \xi
\end{equation}
for all $Y \in TM$, since $M$ has a non-vanishing geodesic Reeb flow, that is, $\alpha = g(S \xi, \xi) \neq 0$ on $M$. It makes a contradiction. In fact, if we substitute $Y=\xi$ in~\eqref{eq 5.2}, then $M$ should have a vanishing geodesic Reeb flow.

\vskip 6pt

Summing up these observations, we assert that
\begin{proposition}\label{proposition 5.1}
There does not exist any real hypersurface in the complex quadric~$Q^{m}$, $m \geq 3$, with the following three conditions:
\begin{enumerate}
\item [\rm (C-1)] {the non-vanishing geodesic Reeb flow,}
\item [\rm (C-2)]{the $\mathfrak A$-isotropic normal unit vector, and }
\item [\rm (C-3)]{the commuting Jacobi operator, that is, $(R_{\xi} {\bar R}_{N}) = ({\bar R}_{N}  R_{\xi})$.}
\end{enumerate}
\end{proposition}

\begin{remark}
\rm
Let $M$ be a Hopf real hypersurface with $\mathfrak A$-isotropic normal vector field~$N$ in $Q^{m}$, $m \geq 3$. By virtue of the proof given in Proposition~\ref{proposition 5.1}, we assert that if $M$ has the vanishing geodesic Reeb flow, then $M$ naturally satisfies the commuting Jacobi operator, $(R_{\xi} {\bar R}_{N}) = ({\bar R}_{N} R_{\xi})$.
\end{remark}

\vskip 17pt

\subsection{Commuting Jacobi operator with $\mathfrak A$-principal unit normal vector field}

Assume that $M$ is a Hopf real hypersurface with $\mathfrak A$-principal unit normal vector field~$N$ in the complex quadric~$Q^{m}$, $m \geq 3$.

\vskip 6pt

The assumption that $N$ is $\mathfrak A$-principal implies that $N$ satisfies $AN=N$ for a complex conjugation~$A \in \mathfrak A$. It yields that a vector field~$A Y$  should be tangent on~$M$ for all $Y \in TM$, because
$$
AY = BY + g(AY, N)N = BY \in TM
$$
(in particular, $A \xi = - AJN = JAN = JN = -\xi \in TM$). From this, the anti-commuting property, $JA = -AJ$, with respect to the complex structure~$J$ and the real structure~$A$ tells us that
\begin{equation}\label{e: 6.1}
\phi A Y = - A \phi Y
\end{equation}
for all $Y \in TM$. By virtue of these properties and \eqref{normal Jacobi operator of M} and \eqref{structure Jacobi op}, the normal Jacobi operator~${\bar R}_{N}$ and the structure Jacobi operator~$R_{\xi}$ are given by respectively
\begin{equation*}
{\bar R}_{N}Y = Y + 2 \eta(Y) \xi + AY
\end{equation*}
and
\begin{equation*}
R_{\xi}Y = Y- 2\eta(Y) \xi - AY  + \alpha SY - \alpha^{2} \eta(Y) \xi.
\end{equation*}
So, the commuting property defined by $(R_{\xi} {\bar R}_{N}) = ({\bar R}_{N} R_{\xi})$ with respect to ${\bar R}_{N}$ and $R_{\xi}$ is equivalent to
\begin{equation}\label{eq 6.1}
\alpha SAY = \alpha ASY,
\end{equation}
because the real structure~$A$ is an anti-linear involution on $TQ^{m}$, that is, $A^{2}= I$.

\vskip 6pt

On the other hand, taking the covariant derivative with respect to $\bar \nabla$ of $AN=N$ along the direction of $Y \in TM$ and using the formula of Weingarten and \eqref{e 3.6} for the $\mathfrak A$-principal unit normal, we have:
\begin{equation*}
\begin{split}
-SY = {\bar \nabla}_{Y}N & =  ({\bar \nabla}_{Y}A)N + A ({\bar \nabla}_{Y}N) \\
& = q(Y) JAN - ASY \\
& = - 2 \alpha \eta(Y) \xi - ASY,
\end{split}
\end{equation*}
that is,
\begin{equation}\label{eq 6.2}
ASY = SY - 2 \alpha \eta(Y) \xi,
\end{equation}
where we have used \eqref{e 3.6} and $({\bar \nabla}_{U}A)V =  q(U) JAV$ for $U$, $V \in T_{[w]}Q^{m}$, $[w] \in Q^{m}$. Taking the symmetric part of \eqref{eq 6.2}, we see that the shape operator~$S$ commutes with the real structure~$A$ on $TM$, that is, $ASY = SAY$ for any $Y \in TM$. From this we get \eqref{eq 6.1}, which is equivalent to $R_{\xi}{\bar R}_{N} = {\bar R}_{N}R_{\xi}$. So we assert the following: {\it Let $M$ be a Hopf real hypersurface in $Q^{m}$, $m \geq 3$. If $M$ has the $\mathfrak A$-principal normal vector field, then the normal Jacobi operator ${\bar R}_{N}$ commutes with the structure Jacobi operator~$R_{\xi}$}.

\vskip 6pt

From this, together with Propositions~\ref{proposition 4.1} and \ref{proposition 5.1}, we can give a complete proof of Theorem~$\rm 1$ in the introduction.

\vskip 17pt

\section{Proof of Theorem~2}\label{section 7}

\setcounter{equation}{0}
\renewcommand{\theequation}{6.\arabic{equation}}
\vspace{0.13in}

Now, we try to classify a Hopf real hypersurface with $\mathfrak A$-principal normal vector field in $Q^{m}$, $m \geq 3$.

\begin{proposition}\label{proposition 7.1}
Let $M$ be a Hopf real hypersurface in $Q^{m}$, $m \geq 3$. If the normal vector~$N$ of $M$ is $\mathfrak A$-principal, then the Reeb curvature function~$\alpha$ is non-vanishing constant on $M$. Moreover, $M$ is a contact real hypersurface with constant mean curvature in $Q^{m}$.
\end{proposition}

\begin{proof}
For any $X$, $Y$, $Z \in \mathcal C$, the Codazzi equation yields
\begin{equation*}
g((\nabla_{X}S)Y - (\nabla_{Y}S)X, Z) =0,
\end{equation*}
so we obtain
\begin{equation}\label{eq: 6.1}
\begin{split}
(\nabla_{X}S)Y - (\nabla_{Y}S)X  & = g((\nabla_{X}S)Y - (\nabla_{Y}S)X, \xi) \xi \\
& = g((\nabla_{X}S)\xi, Y) \xi  - g((\nabla_{Y}S)\xi, X) \xi
\end{split}
\end{equation}
for any $X \in \mathcal C$.

\vskip 3pt

Since $M$ is Hopf, it provides that $(\nabla_{Z}S) \xi = (Z \alpha) \xi + \alpha S \phi Z - S \phi S Z$ for any tangent vector field~$Z$ of $M$. By using this formula, \eqref{eq: 6.1} becomes
\begin{equation}\label{e: 12}
(\nabla_{X}S)Y - (\nabla_{Y}S)X  = g (\alpha S \phi X + \alpha \phi S X - 2 S \phi SX, Y)\xi.
\end{equation}
Moreover, taking the inner product of \eqref{e: 12} with Reeb vector field~$\xi$, and using the equation of  Codazzi, together with the assumption of $\mathfrak A$-principal unit normal,  we have the following for any $X$, $Y \in \mathcal C$
\begin{equation}\label{e: 6.4}
-2g(\phi X, Y) = g(\alpha \phi SX + \alpha S \phi X - 2 S \phi S X , Y).
\end{equation}

\vskip 3pt

On the other hand, since $g(AY, N) =0$ and $g(AY, \xi) =0$ for any $Y \in \mathcal C$, we see that $AY \in \mathcal C$. From \eqref{e: 6.4}, by using \eqref{e: 6.1}, \eqref{eq 6.2}, and \eqref{eq: 3.1}, it follows that
\begin{equation*}
\begin{split}
&-2g(\phi X, AY)  = g(\alpha \phi SX + \alpha S \phi X - 2 S \phi S X , AY) \\
& \quad \Longleftrightarrow \ -2g(A \phi X, Y)  = g(\alpha  A \phi SX + \alpha A S \phi X - 2 A S \phi S X, Y)\\
& \underset{{\eqref{e: 6.1} \& \eqref{eq 6.2}}}{\Longleftrightarrow} \ 2g(\phi A X, Y)  = g(-\alpha  \phi ASX + \alpha S \phi X - 2 S \phi S X, Y) \\
& \quad \underset{\eqref{eq: 3.1}}{\Longleftrightarrow} \ 2g(\phi A X, Y)  = -g(\alpha  \phi ASX + \alpha \phi SX  + 2 \phi X, Y)\\
& \quad \Longleftrightarrow g(\alpha  \phi ASX + \alpha \phi SX + 2 \phi X + 2 \phi AX, Y)=0, \quad \forall \ X, Y \in \mathcal C.
\end{split}
\end{equation*}
Since $\phi Z \in \mathcal C$ for any $Z\in TM$, the last equation becomes
\begin{equation}\label{e: 14}
\alpha  \phi ASX + \alpha \phi SX + 2 \phi X + 2 \phi AX = 0, \quad  \forall X \in \mathcal C.
\end{equation}

\vskip 6pt

Suppose $\alpha = 0$. Then \eqref{e: 14} becomes
\begin{equation}\label{e: 13}
\phi AX = - \phi X, \quad \forall X \in \mathcal C.
\end{equation}
Applying the structure tensor~$\phi$, it yields that
\begin{equation}\label{eq: 6.5}
AX = -X,  \quad  \forall  X \in \mathcal C.
\end{equation}
In addition, by \eqref{e: 6.1} for the anti-commuting $A \phi = - \phi A$, the equation~\eqref{e: 13} becomes $A \phi X = \phi X$ for all $X \in \mathcal C$. Since $\phi X \in \mathcal C$, it follows that $AX=X$ for any $X \in \mathcal C$. Combining this fact and \eqref{eq: 6.5}, we get $X=0$ for all $X \in \mathcal C$, which means that $\mathrm{dim}_{\mathbb R}\,\mathcal C =0$. It makes a contradiction for the dimension of $\mathcal C$. In fact, from the geometric structure of the tangent vector space $T_{[z]}M$ of $M$ at $[z] \in M \subset Q^{m}$ we can take one basis for $T_{[z]}M$ as $\{ \xi, e_{1}, \cdots, e_{2m-2}\} =\{\xi\} \oplus \mathcal C$. From this, we get
$$
\mathrm{dim}_{\mathbb R}\,\mathcal C =2m-2=0,
$$
that is, $m =1$. Accordingly, we assert that the Reeb curvature function~$\alpha$ is non-vanishing on~$M$. Moreover, by Lemmas~\ref{lemma constant} and \ref{lemma 3.2}, it is know that the Reeb curvature function~$\alpha$ is constant on $M$.

\vskip 6pt

Applying to the structure tensor~$\phi$ of \eqref{e: 14}, and using \eqref{eq 6.2} and $\mathfrak A$-principal singular unit normal, it yields
\begin{equation}\label{e: 15}
\alpha SX = -X - AX, \quad \forall X \in \mathcal C.
\end{equation}
Let $X \in \mathcal C$ be a principal vector with corresponding principal curvature~$\lambda$, that is, $SX=\lambda X$.
Then the equation~\eqref{e: 15} leads to
$$
AX = - (\alpha \lambda +1) X.
$$
From this, applying the real structure~$A$ to the left, and using $A^{2}=I$ and the above equation again, we get
$$
\alpha \lambda (\alpha \lambda +2) =0.
$$

\begin{itemize}
\item {Case 1. \quad $\alpha \lambda =0$. }
\end{itemize}

\vskip 3pt

Since $\alpha \neq 0$ on $M$, we obtain $\lambda = 0$. From Lemma~\ref{lemma 3.2}, $\phi X$ is also a principal vector with principal curvature~$\mu = - \frac{2}{\alpha}$. In this case, the expression of the shape operator~$S$ of $M$ is given by
$$
S = \mathrm{diag}(\alpha, \underbrace{0,0,\cdots, 0}_{(m-1)}, \underbrace{- \frac{2}{\alpha}, - \frac{2}{\alpha}, \cdots, - \frac{2}{\alpha}}_{(m-1)}).
$$

\vskip 6pt

\begin{itemize}
\item {Case 2. \quad $\alpha \lambda +2 =0$. }
\end{itemize}

\vskip 3pt

From the assumption of $\alpha \lambda +2 =0$, we see that $\lambda = - \frac{2}{\alpha}$. Moreover, by virtue of Lemma~\ref{lemma 3.2}, $\phi X$ is also a principal vector with the principal curvature~$\mu = 0$. Thus, the expression of the shape operator~$S$ of $M$ is given by
$$
S = \mathrm{diag}(\alpha, \underbrace{- \frac{2}{\alpha}, - \frac{2}{\alpha}, \cdots, - \frac{2}{\alpha}}_{(m-1)}, \underbrace{0,0,\cdots, 0}_{(m-1)})
$$

\vskip 6pt

Summing up the above two cases, it follows that the shape operator~$S$ of $M$ satisfies $S \phi + \phi S = \delta \phi$, where $\delta = - \frac{2}{\alpha}\neq 0$. That is, $M$ becomes a contact real hypersurface in $Q^{m}$. In addition, we obtain that $M$ is a Hopf real hypersurface with constant mean curvature, because the trace of the shape operator~$S$ is given by
$$
\mathrm{Tr}S = \alpha - (m-1)(\frac{2}{\alpha}).
$$
Consequently, it satisfies all the assumptions given in Theorem~$\rm A$. So we can give a complete proof of Proposition~\ref{proposition 7.1}.
\end{proof}

\vskip 6pt

By virtue of Theorem~$\rm A$ due to Berndt and Suh~\cite{BS2015} and Proposition~\ref{proposition 7.1}, we assert that {\it if $M$ is a Hopf real hypersurface with $\mathfrak A$-principal singular normal vector field in $Q^{m}$, $m \geq 3$, then~$M$ is congruent to an open part of the tube around~$S^{m}$ in~$Q^{m}$}. We call this tube a model space of $(\mathcal T_{B})$. Conversely, it was proved that the model space~$(\mathcal T_{B})$ is a Hopf real hypersurface with $\mathfrak A$-principal singular normal vector field in $Q^{m}$ (see \cite{BS2015} and \cite{Suh2017}). Thus we can give a complete proof of Theorem~2 in the Introduction.

\begin{remark}
{\rm
By virtue of Theorem~1, the above result can be rewritten as follows. {\it
Let $M$ be a real hypersurface with non-vanishing geodesic Reeb flow in~$Q^{m}$, $m \geq 3$. Then $M$ has the commuting normal Jacobi operator, ${\bar R}_{N} R_{\xi} = R_{\xi}  {\bar R}_{N}$, if and only if $M$ is locally congruent to the model space of $(\mathcal T_{B})$}.
}
\end{remark}

\vskip 17pt

\section{Proof of Theorem~3}\label{section 8}

\setcounter{equation}{0}
\renewcommand{\theequation}{7.\arabic{equation}}
\vspace{0.13in}

In this section we assume that $M$ is a Hopf real hypersurface  in $Q^{m}$, $m \geq 3$, with the commuting normal Jacobi operator which is different in section~\ref{section 4} or \ref{section 5}.

\vskip 3pt

As mentioned in the introduction, the Jacobi operator~$R_{X}$ with respect to a tangent vector field $X \in TM$ is defined by ${R}_{X}Y := {R}(Y, X)X$ for any $Y \in TM$, and it becomes a self-adjoint endomorphism of the tangent bundle $TM$ of $M$. That is, the Jacobi operator satisfies $R_{X} \in \mathrm{End}(TM)$ and is symmetric in the sense of $g(R_{X}Y, Z)=g(Y, R_{X}Z)$ for any tangent vector fields $Y$ and $Z$ on~$M$.

\vskip 6pt

From now on, assume that the normal Jacobi operator~${\bar R}_{N}$ of $M$ satisfies the new commuting condition given by
\begin{equation}\label{**}
{\bar R}_{N} R_{X} = R_{X} {\bar R}_{N}
\tag{**}
\end{equation}
where $R_{X}$ is the Jacobi operator with respect to $X \in \mathcal C =\{X \in TM \, | \, X \bot \xi \}$. Actually, from the Gauss equation~\eqref{eq: 2.1} of $M$, the Jacobi operator~$R_{X}$ with respect to $X \in \mathcal C$ is defined by
\begin{equation}\label{eq: 8.1}
\begin{split}
R_{X}Y & := R(Y, X)X \\
&\,\,  = g(X,X)Y - g(X,Y)X - 3g(X, \phi Y) \phi X + g(BX, X) BY \\
& \quad \  - g(BX, Y) BX + g( \phi BX, X) \phi BY - g(\phi BX, X) g(AN, Y) \xi  \\
& \quad \ - g(\phi BY, X) \phi BX + g(\phi BY, X) g(AN,  X) \xi \\
& \quad \ + g(SX, X) SY - g(SY, X) SX
\end{split}
\end{equation}
for any $X \in \mathcal C$ and $Y \in TM$. By using this equation, we obtain:
\begin{proposition}\label{proposition 8.1}
Let $M$ be a Hopf real hypersurface in $Q^{m}$, $m \geq 3$. If the normal Jacobi operator~${\bar R}_{N}$ of $M$ commutes with the Jacobi operator~$R_{X}$ for any vector field $X \in \mathcal C$, then the normal vector field~$N$ of $M$ is singular.
\end{proposition}

\begin{proof}
Since ${\bar R}_{N}$ and $R_{X}$ are symmetric, the commuting condition~\eqref{**} yields that
\begin{equation} \label{e: 8.2}
g(R_{X} (A\xi),{\bar R}_{N}\xi) = g({\bar R}_{N}(A\xi), R_{X}\xi)
\end{equation}
for any $X \in \mathcal C$.

\vskip 3pt

\noindent On the other hand, from \eqref{normal Jacobi operator of M} we get the following:
$$
{\bar R}_{N}\xi = 4 \xi - 2 \beta A\xi \ \  \mathrm{and}\ \  {\bar R}_{N}(A\xi)=2 \beta \xi.
$$
In addition, by using \eqref{eq: 8.1}, the formula~\eqref{e: 8.2} leads to
\begin{equation*}
\begin{split}
& 2 g(BX,X) + \alpha \beta g(SX, X) + 2 \beta g(A\xi, X) g(A\xi, X)- 2 \beta g(AN, X) g(AN, X)   \\
& \quad  - 2 \beta^{2} g(BX, X) - \beta g(SX, X) g(SA\xi, A\xi) + \beta g(SA\xi, X) g(SA\xi, X) =0
\end{split}
\end{equation*}
for any $X \in \mathcal C$. This equality also holds for $X+Y$ where $Y \in \mathcal C$. Then it induces
\begin{equation}\label{e: 8.3}
\begin{split}
& 2 g(BX,Y) + \alpha \beta g(SX, Y) + 2 \beta g(A\xi, X) g(A\xi, Y)- 2 \beta g(AN, X) g(AN, Y)   \\
& \quad  - 2 \beta^{2} g(BX, Y) - \beta g(SX, Y) g(SA\xi, A\xi) + \beta g(SA\xi, X) g(SA\xi, Y) =0
\end{split}
\end{equation}
for any $X$, $Y \in \mathcal C$. Now we may put
\begin{equation}\label{e 8.4}
\begin{split}
W & := 2 BX + \alpha \beta SX + 2 \beta g(A\xi, X) A\xi + 2 \beta g(AN, X) \phi A\xi   \\
& \quad  - 2 \beta^{2} BX - \beta g(SA\xi, A\xi) SX  + \beta g(SA\xi, X)SA\xi.
\end{split}
\end{equation}
Then by using of \eqref{e: 8.3} and \eqref{e 8.4}, the vector field~$W \in TM$ can be given by
\begin{equation}\label{e 8.5}
\begin{split}
W & = \sum_{i=1}^{2m-1} g(W, e_{i}) e_{i} = \sum_{i=1}^{2m-2} g(W, e_{i}) e_{i} + g(W, \xi)\xi \\
  & \underset{\eqref{e: 8.3}}{=} g(W, \xi)\xi  \underset{\eqref{e 8.4}}{=} \Big \{2 g(A\xi, X) + \alpha \beta^{2} g(S A\xi, X) \Big \}\xi
\end{split}
\end{equation}
for a basis $\{e_{1}, e_{2}, \cdots, e_{2m-2}, e_{2m-1}=\xi \}$ of $TM$.

\vskip 3pt

\noindent On the other hand, taking the inner product of \eqref{e 8.4} with the vector field~$A\xi$, we have
\begin{equation*}
g(W, A\xi) = \alpha \beta g(SX, A\xi) + 2 \beta g(A\xi, X), \quad \forall X \in \mathcal C,
\end{equation*}
together with $BA\xi = \xi$. Next let us take the inner product of \eqref{e 8.5} with $A\xi$, it follows $g(W, A\xi)= 2 \beta g(A\xi, X) + \alpha \beta^{3} g(S A\xi, X)$. Combining these two equations, we get
\begin{equation}\label{e: 8.4}
\alpha \beta (1-\beta)(1+\beta) g(SA\xi, X) = 0
\end{equation}
for any $X \in \mathcal C$.

\vskip 3pt

As mentioned before, if the Reeb curvature function~$\alpha = g(S\xi, \xi)$ is vanishing, then the normal vector field~$N$ is singular. Moreover, since $\beta = g(A\xi, \xi)= -\cot 2t$ where $t \in [0, \frac{\pi}{4}]$, the normal vector field~$N$ is singular if $\beta=0$ or $\beta=-1$. That is, when $\beta=0$ (resp. $\beta = -1$), the normal vector field~$N$ is $\mathfrak A$-isotropic (resp. $\mathfrak A$-principal). The other case like $\beta=1$ can not be happen.

\vskip 3pt

Finally, let us consider the remained case, $g(SA\xi, X) =0$ for any $X \in \mathcal C$. In other words, it implies that
\begin{equation}\label{e: 8.5}
SA \xi = \sum_{i=1}^{2m-2}g(SA\xi, e_{i}) e_{i} + g(SA\xi, \xi)\xi = g(SA\xi, \xi)\xi = \alpha \beta \xi
\end{equation}
with $\alpha \beta (1-\beta^{2}) \neq 0$ on $M$. By putting $X = A \xi $ in \eqref{eq: 3.1} and \eqref{e: 8.5}, we get
\begin{equation}\label{e: 8.6}
S \phi A \xi = \sigma \phi A \xi \quad \mathrm{where}\  \sigma = -{2 \beta^{2}} / {\alpha}.
\end{equation}

On the other hand, from ${\bar R}_{N}(\phi A\xi)=0$ and \eqref{eq: 8.1}, the commuting condition $g(R_{X} (\phi A \xi), {\bar R}_{N} \xi) = g(R_{X} \xi, {\bar R}_{N} (\phi A \xi))$ becomes
\begin{equation*}
\begin{split}
& 4 \beta g(X, AN) g(A\xi, X) - 4 \beta^{2} g(\phi BX, X) + 4 g(\phi BX, X) \\
& \quad -  2 \beta^{2} g(A\xi, X) g(BX, AN) - 2 \beta^{3} g(A\xi, X) g(AN, X) \\
& \quad - 2 \beta g(SX, X) g(S\phi A\xi, A\xi) + 2 \beta g(S\phi A\xi, X) g(SX, A\xi) =0
\end{split}
\end{equation*}
for $X \in \mathcal C$. By using the polarization of the inner product, we get
\begin{equation}\label{e: 8.7}
\begin{split}
& 4 \beta g(X, AN) g(A\xi, Y) + 4 \beta g(Y, AN) g(A\xi, X) - 4 \beta^{2}g(\phi BX, Y)  \\
& \quad - 4 \beta^{2} g(\phi BY, X) + 4 g(\phi BX, Y) + 4 g(\phi BY, X) \\
& \quad - 2 \beta^{2} g(A\xi, Y) g(BX, AN)-  2 \beta^{2} g(A\xi, X) g(BY, AN) \\
& \quad - 2 \beta^{3} g(A\xi, Y) g(AN, X) - 2 \beta^{3} g(A\xi, X) g(AN, Y) \\
& \quad - 2 \beta g(SX, Y) g(S\phi A\xi, A\xi) - 2 \beta g(SY, X) g(S\phi A\xi, A\xi)\\
& \quad + 2 \beta g(S\phi A\xi, X) g(SY, A\xi) + 2 \beta g(S\phi A\xi, Y) g(SX, A\xi)=0
\end{split}
\end{equation}
for $X$, $Y \in \mathcal C$. By the way, from the property of $JA=-AJ$ we have $\phi BZ - g(AN,Z) \xi = - B\phi Z + \eta(Z) \phi A \xi$ for any $Z \in TM$. Hence $B \phi A\xi = \beta \phi A \xi$, because $\phi BA\xi = \phi \xi$. From this and putting $Y=A\xi \in \mathcal C$ in \eqref{e: 8.7} it follows
\begin{equation*}
\beta \big(2-2\beta^{2}+\beta^{4}\big) g(AN, X) =0 \quad \mathrm{for}\ X \in \mathcal C,
\end{equation*}
together with \eqref{e: 8.5} and \eqref{e: 8.6}. Since $\beta (2-2\beta^{2}+\beta^{4})\neq 0$, it implies that
\begin{equation}\label{e: 8.8}
g(AN, X)=-g(\phi A\xi, X)=0
\end{equation}
for all $X \in \mathcal C$. Then it follows that
\begin{equation*}
\begin{split}
\phi A \xi & = \sum_{i=1}^{2m-2} g(\phi A \xi, e_{i})e_{i} + g(\phi A\xi, \xi)\xi =0,
\end{split}
\end{equation*}
which implies $A \xi = \beta \xi$. This means that $\beta^{2}=1$, which gives us a contradiction.

\vskip 3pt

Making use of these facts, we have only the case that $\alpha \beta (1+\beta)=0$. So we conclude that the normal vector field~$N$ of $M$ is singular. So we are able to give a complete proof of Proposition~\ref{proposition 8.1}.

\end{proof}

For the latter part of the proof of Theorem~$\rm 3$ we can divide into two cases that the unit normal vector field~$N$ is either $\mathfrak A$-isotropic or $\mathfrak A$-principal. Thus as a first part we consider the case of $\mathfrak A$-isotropic as follows.
\begin{proposition}\label{proposition 8.2}
There does not exist a Hopf real hypersurfce with $\mathfrak A$-isotropic normal vector field~$N$ and the commuting condition~\eqref{**} in $Q^{m}$, $m \geq 3$.
\end{proposition}

\begin{proof}
Suppose that the normal vector field~$N$ of $M$ is $\mathfrak A$-isotropic, that is, $\beta = 0$. Let $\mathcal Q$ is a distribution of $TM$ defined by
$$
\mathcal Q_{[z]} = \mathcal C_{[z]} - [A\xi, AN]_{[z]} = \{Z \in T_{[z]}M \, | \, Z \bot \xi, A\xi, AN \} \quad \mathrm{at}\ [z] \in M.
$$
Since ${\bar R}_{N}\xi = 4 \xi$ and ${\bar R}_{N} Z = Z$ for any $Z \in \mathcal Q$, the commuting condition~\eqref{**}, $g({\bar R}_{N}R_{X} \xi, Z) = g(R_{X} {\bar R}_{N} \xi, Z)$, gives us $g( R_{X}\xi, Z ) = 0$ for all $Z \in \mathcal Q$. Thus, by using \eqref{eq: 8.1} we get
\begin{equation*}
-g(A\xi,X) g(BX, Z) + g(AN, X) g(\phi BX, Z) =0.
\end{equation*}
By virtue of the linearity in terms of the inner product, it is equal to
\begin{equation}\label{eq: 8.4}
\begin{split}
&- g(A\xi, Y) g(BX, Z) - g(A\xi, X) g(BY, Z) \\
& \quad \ \  + g(AN, Y) g(\phi BX, Z) + g(AN, X) g(\phi BY, Z)=0
\end{split}
\end{equation}
for any $X$, $Y \in \mathcal C$ and $Z \in \mathcal Q$. Then for any basis $\{e_{1}, e_{2}, \cdots, e_{2m-4}, e_{2m-3}=A\xi, e_{2m-2}=AN, e_{2m-1}=\xi \}$ of $TM$, the equation~\eqref{eq: 8.4} yields that
\begin{equation*}
\begin{split}
{\widetilde W}&= \sum_{i=1}^{2m-1} g(W, e_{i}) e_{i} \\
  & = \sum_{i=1}^{2m-2} g({\widetilde W}, e_{i}) e_{i} + g({\widetilde W}, A\xi)A\xi + g({\widetilde W}, AN) AN + g(W, \xi)\xi \\
  & = g({\widetilde W}, A\xi)A\xi + g({\widetilde W}, AN) AN + g(W, \xi)\xi
\end{split}
\end{equation*}
where ${\widetilde W}= - g(A\xi, Y) BX - g(A\xi, X) BY + g(AN, Y) \phi BX + g(AN, X) \phi BY$. Since $BAN= - B \phi A \xi =0$ and $BA\xi = \xi$, it follows
\begin{equation}\label{eq: 8.5}
\begin{split}
2 g(A\xi, X) g(A\xi, Y)\xi & =  g(A\xi, Y) BX + g(A\xi, X) BY  \\
 & \quad - g(AN, Y) \phi BX - g(AN, X) \phi BY
\end{split}
\end{equation}
for any $X$, $Y \in \mathcal C$. Since $\beta = g(A\xi, \xi)=0$, we know that $A\xi \in \mathcal C$. Hence substituting $Y=A\xi$ in \eqref{eq: 8.5}, it leads to $BX = g(A\xi, X) \xi$ for $X \in \mathcal C$. Therefore, we obtain
\begin{equation}\label{eq: 8.6}
B^{2}X = g(A\xi, X) A\xi
\end{equation}
because of $A\xi = B\xi$. In general, from the properties of $AJ = -JA$ and $A^{2}=I$, we get
$$
AN = AJ\xi = -JA\xi = -\phi A \xi - g(A\xi, \xi) N,
$$
and
$$
B^{2}Y = Y + g(AN, Y) \phi A \xi
$$
for any $Y \in TM$. From this, \eqref{eq: 8.6} becomes
\begin{equation*}
X = g(A\xi, X) A\xi + g(AN, X) AN \quad \mathrm{for \  any}\ X \in \mathcal C =[A\xi, AN] \oplus \mathcal Q,
\end{equation*}
which implies $\mathrm{dim}_{\mathbb R}\,{\mathcal C} = 2$. It makes a contradiction for $m \geq 3$, which gives a complete proof of Proposition~\ref{proposition 8.2}.
\end{proof}

By virtue of Propositions~\ref{proposition 8.1} and \ref{proposition 8.2}, we can assert that {\it the normal vector field~$N$ of $M$ satisfying the condition of \eqref{**} in $Q^{m}$ must be $\mathfrak A$-principal}. Accordingly, by using Theorem~$\rm 2$, we arrive at the conclusion that
\begin{quote}
{\it Let $M$ be a Hopf real hypersurface in $Q^{m}$, $m \geq 3$. If the normal Jacobi operator~${\bar R}_{N}$ commutes with the Jacobi operator~$R_{X}$ with respect to $X \in \mathcal C$, then $M$ is locally congruent to the model space of type~$(\mathcal T_{B})$.}
\end{quote}

\vskip 6pt

Now, in order to prove our Theorem~$\rm 3$, it remains only to check whether the model space of $(\mathcal T_{B})$ satisfy the commuting condition~\eqref{**}. According to Remark~5.1. in \cite{Suh2017}  we obtain:
\begin{proposiotion A}
Let $(\mathcal T_{B})$ be the tube of radius $0 < r < \frac{\pi}{2 \sqrt{2}}$ around the $m$-dimensional sphere $S^{m}$ in $Q^{m}$. Then the following holds:
\begin{enumerate}[\rm (i)]
\item {$(\mathcal T_{B})$ is a Hopf hypersurface.}
\item {The normal bundle of $(\mathcal T_{B})$ consists of $\mathfrak A$-principal singular.}
\item {$(\mathcal T_{B})$ has three distinct constant principal curvatures.
\begin{center}
\begin{tabular}{l|l|l}
\hline
\mbox{\rm principal curvature} & \mbox{\rm eigenspace}  & \mbox{\rm multiplicity}  \\
\hline
$\alpha = -\sqrt{2}\cot(\sqrt{2}r)$ & $T_{\alpha}=\mathrm{Span}\{\xi\}$ & $1$ \\
$\lambda = \sqrt{2} \tan(\sqrt{2}r)$ & $T_{\lambda}=V(A) \cap {\mathcal C}=\{X \in \mathcal C\,|\, AX=X\}$ & $m-1$\\
$\mu = 0$ & $T_{\mu} = JV(A) \cap {\mathcal C}=\{X \in \mathcal C\,|\, AX=-X\}$ & $m-1$ \\
\hline
\end{tabular}
\end{center}}
\item {$S \phi + \phi S = 2 \delta \phi $, $\delta =-\frac{1}{\alpha}\neq 0$ (contact hypersurface). }
\end{enumerate}
\end{proposiotion A}

\noindent By virtue of $\rm (ii)$ in Proposition~$\rm A$, we know that $AN=N$. So it follows that $AY$ is a tangent vector field of $(\mathcal T_{B})$ for any $Y \in T (\mathcal T_{B})$. Thus from \eqref{eq: 2.1} and \eqref{normal Jacobi operator of M} the Jacobi operators with respect to $N$ and~$X$, respectively, are given by
\begin{equation*}
{\bar R}_{N}Y = Y + 2 \eta(Y) \xi + AY
\end{equation*}
and
\begin{equation*}
\begin{split}
R_{X}Y & = g(X,X)Y - g(X,Y)X + 3g(\phi X, Y) \phi X + g(AX, X)AY \\
& \quad \  - g(AX, Y) AX  + g(\phi AX, X) \phi AY + g(A \phi X, Y) \phi AX \\
& \quad \  + g(SX, X) SY - g(SX, Y) SX
\end{split}
\end{equation*}
for all $X$, $Y \in T(\mathcal T_{B})$. Then we see that all tangent vector fields are principal by~${\bar R}_{N}$, that is,
\begin{equation}
{\bar R}_{N}Y = \left\{ \begin{array}{cl}
2 Y & \mathrm{if}\  Y \in T_{\alpha}\\
2 Y & \mathrm{if}\  Y \in T_{\lambda}\\
0 & \mathrm{if} \ Y \in T_{\mu}
\end{array}\right.
\end{equation}
where $T(\mathcal T_{B}) = T_{\alpha} \oplus T_{\lambda} \oplus T_{\mu}$. On the other hand, the Jacobi operator with respect to $X \in T(\mathcal T_{B})$ can be expressed by the following three cases.
\begin{enumerate}
\item[\bf Case~1.] {$X \in T_{\alpha}$ (that is, $X=\xi$)}
\begin{equation}
\begin{split}
R_{\xi}Y & = Y - 2\eta(Y) \xi -AY +\alpha SY - \alpha^{2} \eta(Y)\xi \\
& = \left\{ \begin{array}{ll}
0 & \mathrm{if}\  Y \in T_{\alpha}\\
\alpha \lambda Y & \mathrm{if}\  Y \in T_{\lambda}\\
2Y & \mathrm{if} \ Y \in T_{\mu}
\end{array}\right.
\end{split}
\end{equation}
\item[\bf Case~2.] {$X \in T_{\lambda}=\{X \in \mathcal C\, |\, AX = X \}$}
\begin{equation}
\begin{split}
R_{X}Y & = g(X, X) Y - 2g(X, Y) X + 3g(\phi X, Y) \phi X + g(X, X) AY \\
& \quad \  + g(A \phi X, Y) \phi X + \lambda g(X,X) SY - \lambda^{2} g(X,Y) X \\
& = \left\{ \begin{array}{ll}
\alpha \lambda g(X,X) Y & \mathrm{if}\  Y \in T_{\alpha}\\
(\lambda^{2}+2) g(X,X)Y - (\lambda^{2}+2)g(X,Y) X & \mathrm{if}\  Y \in T_{\lambda}\\
2g(\phi X, Y) \phi X & \mathrm{if} \ Y \in T_{\mu}
\end{array}\right.
\end{split}
\end{equation}
\item[\bf Case~3.] {$X \in T_{\mu}=\{X \in \mathcal C\, |\, AX = -X \}$}
\begin{equation}
\begin{split}
R_{X}Y & = g(X,X) Y - 2g(X,Y) X +3g(\phi X, Y) \phi X -g(X,X) AY \\
& \quad \  -g(A \phi X, Y) \phi X \\
& = \left\{ \begin{array}{ll}
2g(X,X)Y & \mathrm{if}\  Y \in T_{\alpha}\\
2g(\phi X, Y) \phi X & \mathrm{if}\  Y \in T_{\lambda}\\
2g(X,X) Y - 2g(X, Y) X & \mathrm{if} \ Y \in T_{\mu}
\end{array}\right.
\end{split}
\end{equation}
\end{enumerate}
where we have used that $\phi Z \in T_{\mu}$ (resp. $\phi Z \in T_{\lambda}$), provided that $Z \in T_{\lambda}$ (resp. $Z \in T_{\mu}$).
From these equations we consequently obtain:
\begin{equation*}
\begin{split}
 {\bar R}_{N}R_{X}Y & =  \left\{ \begin{array}{ll}
0 & \mathrm{if}\  X \in T_{\alpha}, \ Y \in T_{\alpha}\\
2\alpha \lambda Y & \mathrm{if}\ X \in T_{\alpha}, \  Y \in T_{\lambda}\\
0 & \mathrm{if} \ \mathrm{if}\ X \in T_{\alpha}, \  Y \in T_{\mu} \\
2\alpha \lambda g(X,X) Y & \mathrm{if}\  X \in T_{\lambda},\ Y \in T_{\alpha}\\
2(\lambda^{2}+2) g(X,X) Y - 2(\lambda^{2}+2) g(X,Y) X & \mathrm{if}\ X \in T_{\lambda}, \  Y \in T_{\lambda}\\
0 & \mathrm{if} \ \mathrm{if}\ X \in T_{\lambda}, \  Y \in T_{\mu} \\
4g(X,X) Y & \mathrm{if}\  X \in T_{\mu},\ Y \in T_{\alpha}\\
4g(\phi X, Y) \phi X & \mathrm{if}\ X \in T_{\mu}, \  Y \in T_{\lambda}\\
0 & \mathrm{if} \ \mathrm{if}\ X \in T_{\mu}, \  Y \in T_{\mu} \\
\end{array}\right. \\
& =R_{X}{\bar R}_{N}Y
\end{split}
\end{equation*}
It implies that the model space of $(\mathcal T_{B})$ satisfies the commuting condition~\eqref{**} between the normal Jacobi operator~${\bar R}_{N}$ and the Jacobi operator~$R_{X}$ for $X \in \mathcal C$. Moreover, in the above calculations it can be easily checked that ${\bar R}_{N}R_{\xi} = R_{\xi} {\bar R}_{N}$ for the tube of $\rm (\mathcal T_{B})$.

\vskip 17pt

\begin{acknowledgements}
\rm The first author was supported by grant Proj. No. NRF-2019-R1I1A1A01-050300 and the second author by NRF-2018-R1D1A1B05-040381 from National Research Foundation of Korea.
\end{acknowledgements}

\vskip 17pt


\end{document}